\documentclass[12pt]{amsart}
 
\usepackage{graphicx} 
\usepackage{amsmath}

\usepackage[normalem]{ulem}

\usepackage{amsfonts}
\usepackage{amsthm}
\usepackage{color}
\usepackage{mathtools}

\usepackage{hyperref}
\hypersetup{
    colorlinks=true,
    citecolor=red,
    linkcolor=blue,
    filecolor=magenta,      
    urlcolor=red,
   }

\usepackage[a4paper, margin=1.3in]{geometry}
\usepackage{tikz}
\usepackage{tikz-cd}
\usepackage{enumitem}
\usepackage{xcolor}
\usepackage{amssymb}

\usepackage{chngcntr}

\newtheorem{theorem}{Theorem}[section]
\newtheorem{proposition}[theorem]{Proposition}
\newtheorem{lemma}[theorem]{Lemma}
\newtheorem{corollary}[theorem]{Corollary}

\theoremstyle{definition}
\newtheorem{example}[theorem]{Example}
\newtheorem{definition}[theorem]{Definition}

\newcommand\restr[2]{{
  \left.\kern-\nulldelimiterspace 
  #1
  \vphantom{\big|} 
  \right|_{#2} 
  }}

\newcommand{\bigO}{\mathcal{O}}

\newcommand{\folF}{\mathcal{F}}

\DeclareSymbolFont{AMSb}{U}{msb}{m}{n}

\author{Alessandro Passantino}
\title{Numerical conditions for the boundedness of foliated surfaces}

\date{\today}

\begin{document}

\begin{abstract}
We show that the set of Hilbert functions $P(m)=\chi(mK_\folF)$ of 2-dimensional foliated canonical models with fixed $K_\folF^2$, $K_\folF \cdot K_X$ and $i_\mathbb{Q}(\folF)$ is finite. As a consequence, we deduce that two results on the effective birationality and boundedness of foliated canonical models with fixed Hilbert function still hold when only $K_\folF^2$, $K_\folF \cdot K_X$ and $i_\mathbb{Q}(\folF)$ are fixed. We then give examples further investigating the properties of families of canonical models, and study particular cases in which some of the conditions are not necessary. 
\end{abstract}

\address{Università degli Studi di Pavia, Dip. di Matematica, Via Ferrata 5, 27100 Pavia, Italy.}
\email{a.passantino@campus.unimib.it}

\maketitle

\thispagestyle{empty}

\section{Introduction}
Throughout this paper, we work over the field $\mathbb{C}$ of complex numbers.

In recent years, the theory of foliations in algebraic geometry has seen increasing attention from a birational point of view: after works of Brunella, McQuillan, Mendes and others (see \cite{mcquillan08} and \cite{bru15} for an in-depth survey), it has become clear that the birational geometry of a mildly singular foliation $\folF$ (as a saturated subsheaf of $T_X$, closed under Lie bracket) on a variety $X$ is governed by the first Chern class of $\folF$. In view of this, it is natural to investigate the numerical properties of canonical divisors of foliations, in analogy with the study of canonical divisors on singular normal varieties. For example, paralleling the classification of projective varieties, it is known that any smooth foliated surface $(X,\folF)$ with foliated canonical singularities and pseudoeffective canonical divisor $K_\folF$ admits a unique minimal model $(X_m,\folF_m)$ with $K_{\folF_m}$ nef and only foliated canonical singularities. On the other hand, when the focus is restricted to foliations of general type, there are notable differences between the properties of canonical divisors of foliations and varieties: for instance, if the goal is to study families of foliated surfaces, it is worth mentioning that foliations of general type do not always admit canonical models in the usual sense, that is, a model with only foliated canonical singularities such that $K_\folF$ is ample. In fact, McQuillan has shown that at cusp singularities, which are foliated canonical singularities, $K_\folF$ is not $\mathbb{Q}$-Cartier. This issue leads to some substantial differences from the classical approach to varieties of general type, and motivates a weaker notion of canonical model of a foliated surface (as in \cite{mcquillan08}), where $K_\folF$ is only numerically ample if $K_\folF$ is not $\mathbb{Q}$-Cartier (Definition \ref{def:canmod}). Since the projectivity of such models is not known, this approach requires to work in the more general setting of algebraic spaces. The lack of a natural polarisation in the non-$\mathbb{Q}$-Cartier case poses one of the main difficulties towards the construction of bounded families of foliated canonical models, and of a moduli functor describing them.

In this regard, two recent works investigate the properties of families of canonical models: the first, due to Hacon and Langer \cite{langer21}, gives effective birationality for canonical models with given Hilbert function $P(m)=\chi(mK_\folF)$. 

\begin{theorem}[{\cite[Theorem 4.3]{langer21}}]
\label{thm:langer}
Let $P \colon \mathbb{Z}_{\geq 0} \rightarrow \mathbb{Z}$ be an integer-valued function. Then, there exists a constant positive integer $N_P$, only depending on $P(m)$, such that for any canonical model $(X,\folF)$ with $\chi(mK_\folF)=P(m)$ and any integer $M\geq N_P$, $|MK_\folF|$ defines a birational map.
\end{theorem}

Building on this, Chen \cite{chen21} shows that if a particular type of partial resolutions of canonical models is considered, then such resolutions belong to a bounded family if their Hilbert function is fixed.

\begin{theorem}[{\cite[Theorem 3.4]{chen21}}]
\label{thm:chen}
Let $S_P$ be the set of minimal partial du Val resolutions of canonical models $(X,\folF)$ of general type (Definition \ref{def:mpdvr}) with fixed Hilbert function $ \chi(mK_{\folF})=P(m)$. Then, $S_P$ is bounded.
\end{theorem}

A relevant issue with assuming that the Hilbert function is fixed is that, in practice, it is generally unfeasible to compute the Hilbert function of a canonical model. While there is a way to generalise the Riemann-Roch theorem on surfaces to non-Cartier divisors (Theorem \ref{thm:rr}), the formula introduces a sum of terms, related to the non-Cartier points of the divisor, which are in general hard to compute. Although there is a complete description of such terms for the canonical divisor of a foliated canonical model (Proposition \ref{prop:lansings}), calculating the Hilbert function still requires knowing the geometry of the surface and foliation at the singular points. It is then natural to ask whether Theorems \ref{thm:langer} and \ref{thm:chen} still hold under purely numerical assumptions on $K_\folF$ and $K_X$; a positive answer to this question is the main result of this work.

\begin{theorem}[=Theorem \ref{theorem:main}]
\label{theorem:mainintro}
Let $k_1,k_2$ be rational numbers, $s$ a positive integer. Let $\mathcal{H}_{k_1,k_2,s}$ be the set of Hilbert functions $P(m) = \chi(X, mK_\folF)$ of canonical models $(X,\folF)$ of general type such that $K_\folF^2=k_1$, $K_\folF \cdot K_X=k_2$ and $i_\mathbb{Q}(\folF)=s$. Then $\mathcal{H}_{k_1,k_2,s}$ is finite.
\end{theorem}

In other words, the Hilbert function of canonical models is for the most part determined by $K_\folF^2$, $K_\folF \cdot K_X$ and $i_\mathbb{Q}(\folF)$ (the index of $K_\folF$ at points where it is $\mathbb{Q}$-Cartier). In particular, Theorems \ref{thm:langer} and \ref{thm:chen} still hold when only $K_\folF^2$, $K_\folF \cdot K_X$ and $i_\mathbb{Q}(\folF)$ are fixed (Corollaries \ref{cor:hl} and \ref{cor:chen}).

We then focus on the $\mathbb{Q}$-Cartier case and try to understand the properties of families of canonical models, in order to study the sharpness of the assumptions of Theorem \ref{theorem:mainintro}, or different conditions under which it still holds. It is known that fixing $K_\folF \cdot K_X$ is necessary, because there exists an unbounded family of algebraically integrable foliations on smooth surfaces which are canonical models and $K_\folF^2$ is fixed (Example \ref{ex:kf2}). We also give an example of a collection of algebraically integrable foliations, belonging to a bounded family, whose leaves have unbounded genus. Lastly, we further study the condition on $i_\mathbb{Q}(\folF)$: although it appears to be necessary, we are still able to show that it is redundant when we restrict the possible type of underlying surfaces.

\begin{theorem}[=Theorems \ref{thm:i1} and \ref{thm:i2}]
Fix rational numbers $k_1,k_2$. Then, the $\mathbb{Q}$-Gorenstein canonical models $(X,\folF)$ of general type with $K_\folF^2 = k_1, K_\folF \cdot K_X= k_2$, whose underlying surface belongs to one of the following, are bounded.
\begin{enumerate}[label=\textnormal{(\roman*)}]
\item{Surfaces with $-K_X$ big and nef.}
\item{Surfaces with $K_X \equiv 0$.}
\item{Algebraically integrable foliations induced by fibrations with reduced fibers.}
\end{enumerate}
\end{theorem}

Finally, it is worth remarking that McQuillan's canonical models should be compared with $\epsilon$-adjoint canonical models (\cite{spicer22}, generalised further in \cite{chl24}), which are mildly singular foliated surfaces with $K_\folF+\epsilon K_X$ ample; while they are in general distinct, there are cases where the two notions overlap (for instance, $\mathbb{Q}$-Gorenstein canonical models $(X,\folF)$ such that $K_X \equiv 0$). Unlike McQuillan's models, $\epsilon$-adjoint canonical models are projective and they are bounded when only $(K_\folF+ \epsilon K_X)^2$ is fixed (compare this with the boundedness of canonical models of varieties with fixed volume), but by definition involve a different type of singularities (for example, non-$\mathbb{Q}$-Cartier singularities do not appear in $\epsilon$-adjoint models). Similarities should also be noted with minimal partial Du Val resolutions (Definition \ref{def:mpdvr}), which are used for the main proof of this paper.

\medskip

\noindent \textbf{Acknowledgements.} The author would like to thank Calum Spicer, Roberto Svaldi and Luca Tasin for many useful discussions and comments on the content and exposition of the paper.

\section{Preliminaries}

In the following, a surface is always meant to be a 2-dimensional reduced and irreducible algebraic space over $\mathbb{C}$, unless stated otherwise.

\subsection{Intersection theory on normal surfaces}
\label{subsection:inters}
On complete normal surfaces, there is an intersection pairing on Weil divisors (due to Mumford) which generalises the usual intersection of Cartier divisors (see \cite[Section 1]{sakai84} for a further explanation).

Let $X$ be a complete normal surface, $f \colon Y \rightarrow X$ a proper birational morphism from a smooth surface $Y$, and $E = \sum E_i$ the exceptional divisor of $f$. Then, since the matrix $(E_i \cdot E_j)$ is negative definite, for any $\mathbb{R}$-divisor $D$ there exist unique real numbers $x_i$ such that, for all exceptional curves $E_j$, $(f_*^{-1}D+\sum{x_i E_i} )\cdot E_j = 0$ with $f_*^{-1}D$ being the strict transform of $D$; define $f^*D$ to be $f_*^{-1}D+\sum{x_i E_i}$. Then, given two $\mathbb{R}$-divisors $D_1,D_2$, their intersection is defined as $$D_1 \cdot D_2 = f^*D_1 \cdot f^* D_2,$$ where the latter is the usual intersection of $\mathbb{R}$-divisors on a smooth surface. In a similar fashion, it is possible to define the pullback $f^*D$ of any $\mathbb{R}$-divisor whenever $f$ is only a birational morphism of normal surfaces. A Weil $\mathbb{R}$-divisor $D$ is nef if $D \cdot C \geq 0$ for any irreducible curve $C$.

Under the previous pairing, the Hodge index theorem still holds for normal surfaces. We will require it in the following version.

\begin{lemma}[{\cite[Lemma 1.3]{langer21}}]
\label{lemma:hodge}
Suppose that $D_1,D_2$ are $\mathbb{R}$-divisors on a normal surface $Y$ such that $(a_1D_1+a_2D_2)^2>0$ for some $a_1,a_2 \in \mathbb{R}$. Then,
$$D_1^2D_2^2 \leq (D_1 \cdot D_2)^2,$$
where equality holds if and only if there is a non-zero linear combination of $D_1$ and $D_2$ numerically equivalent to 0.
\end{lemma}

\subsection{Riemann-Roch theorem for normal surfaces}
\label{section:rr}
On non-singular surfaces, the Riemann-Roch theorem is fundamental to study the Euler characteristic and Hilbert polynomial of a divisor. This can still be done on singular surfaces, using a more general form of the theorem.

\begin{theorem}[\cite{reid87}, \cite{langer00}]
\label{thm:rr}
Let $X$ be a complete normal surface, $D$ a Weil divisor on $X$.
Then,
$$\chi(X, D)= \frac{1}{2}(D^2 -K_X \cdot D) + \chi(X,\bigO_X) + \sum_{x \in \mathrm{Sing}X} a(x,D),$$
where each $a(x,D)$ depends only on the local isomorphism class of the reflexive sheaf $\mathcal{O}_X(D)$ at $x$.
\end{theorem}

As a particular case, if $D$ is Cartier at $x$ then $a(x,D)=0$: this recovers the formula of the classical Riemann-Roch theorem.

In general, computing $a(x,D)$ is not simple. Still, for the scope of this work, where $D$ is the canonical divisor of a foliation, a complete description is known (Proposition \ref{prop:lansings}).

\subsection{The Kollár-Matsusaka Theorem}
For any big and semiample Cartier divisor $D$ on a normal projective variety, there exists an upper bound on $h^0(mD)$ only depending on the top two coefficient of $P(m)=\chi(X,mD)$, namely $D^n$ and $D^{n-1} \cdot K_X$: this is a consequence of the following theorem.

\begin{theorem} [{\cite[Theorem 2]{km83}}]
\label{theorem:km}
Let $X$ be a normal projective variety of dimension $n$, $D$ a big and semiample Cartier divisor. Then there is a polynomial $Q(m)$ of degree $n-1$, uniquely determined by $D^n$ and $K_X \cdot D^{n-1}$, such that 
$$|h^0(X, mD) - \frac{D^n m^n}{n!}| \leq Q(m).$$
\end{theorem}

When $D$ is ample and Cartier, this  translates to an upper bound on its Hilbert polynomial, thanks to the vanishing of higher degree cohomology groups of $mD$ for large values of $m$.

\subsection{Foliations}

We now give some definitions and properties of foliations and their singularities. A more detailed description of the birational point of view for foliations can be found in \cite{bru15} and \cite{mcquillan08}.

A \emph{foliation} $\folF$ of rank 1 on a normal surface $X$ is given by a rank 1 subsheaf $T_\folF$ of the tangent sheaf $T_X$ of $X$ which is saturated, that is, $T_X/T_\folF$ is torsion free. A singular point of $\folF$ is either a singular point of $X$ or a point at which $T_X/T_\folF$ is not locally free. A \emph{foliated surface} is a pair $(X,\folF)$ of a surface $X$ and a foliation $\folF$ on $X$.

    A Weil divisor $K_\folF$ such that $\bigO(-K_\folF)\simeq T_\folF$ is called the \emph{canonical divisor} of $\folF$.
    
The \emph{Kodaira dimension} of $\folF$ is given by
    $$\kappa(\folF) \coloneqq \kappa(K_\folF) = \max \{\dim \phi_{mK_\folF}(X) \mid m \in \mathbb{N}\},$$
    where $\phi_{mK_\folF}$ is the $m$-th pluricanonical map induced by $mK_\folF$. If $h^0(mK_\folF)=0$ for any $m>0$, we set $\kappa(\folF) = -\infty$. When $\kappa(\folF)$ is maximal, we say that $\folF$ is of general type.
\\

Given a dominant rational map $f \colon Y \dashrightarrow{} X$ and a foliation $\folF$ of rank $r$ on $X$, it is possible to define a \emph{pullback foliation} $f^*\folF$ on $Y$, as in \cite[Section 3.2]{druel}.

For a birational map $g \colon Y \xrightarrow{} X$, the \emph{pushforward} foliation $g_*(\folF)$ of a foliation $\folF$ on $Y$ is given by $g_*(\folF)=(g^{-1})^*\folF$.

Given a dominant rational map $f \colon Y \xrightarrow{} X$, the pullback foliation of the foliation by points on $X$ (that is, $T_\folF = 0$) is called the foliation \emph{induced by} $f$. A foliation induced by a dominant rational map is said to be \emph{algebraically integrable}. When $f \colon Y \rightarrow X$ is an equidimensional morphism and $\mathcal{F}$ is the foliation induced by $f$, we have
$$K_\folF \sim K_{Y/X} - R(f),$$
where $R(f)$ is the \emph{ramification divisor} of $f$, defined as
$$R(f) = \sum_D (f^*D-f^{-1}(D)),$$
the sum running through all the prime divisors on $X$. In particular, if the fibers of $f$ are reduced, then $K_\folF=K_{Y/X}$.

\begin{definition}
A collection $\mathcal{S}$ of projective foliated surfaces is \emph{bounded} if there exists a projective morphism of quasi-projective varieties of finite type $f \colon \mathcal{X} \rightarrow \mathcal{T}$, together with a rank one foliation $\folF$ on $\mathcal{X}$ such that 
\begin{itemize}
    \item{$\folF \subset T_{\mathcal{X}/\mathcal{T}}$;}
    \item{ for any $(X,\folF) \in \mathcal{S}$ there exists $t \in \mathcal{T}$ such that $(X,\folF) \cong (\mathcal{X}_t, \folF_t)$, where $\folF_t = (\restr{\folF}{\mathcal{X}_t})^{**}$.}
    \end{itemize}

\end{definition}

\subsection{Canonical singularities and the Riemann-Roch theorem}
\label{subsection:cansings}
Let $(X,\folF)$ be a foliated normal surface, and $p \colon Y \rightarrow X$ a proper birational morphism. For any divisor $E$ on $Y$, let $a(E,X,\folF)=\text{ord}_E(K_{p^*\folF}-f^*K_\folF)$ be the \emph{discrepancy} of $\folF$ along $E$. We say that a point $x \in X$ is a \emph{canonical} (resp. \emph{terminal}) singularity of $(X,\folF)$ if $a(E,X,\folF) \geq 0$ (resp. $>0$) for every divisor $E$ over $X$.

If $K_\folF$ is Cartier (resp. $\mathbb{Q}$-Cartier) at a point $x$, we say that $\folF$ is \emph{Gorenstein} (resp. $\mathbb{Q}$-Gorenstein) at $x$, or equivalently that $x$ is a Gorenstein (resp. $\mathbb{Q}$-Gorenstein) point of $(X,\folF)$. 
    The \emph{index} $i(\folF)$ of a foliation $\folF$ is the smallest positive integer $m$ such that $mK_\folF$ is Cartier (we set $i(\folF)=\infty$ if $\folF$ is not $\mathbb{Q}$-Gorenstein). The $\mathbb{Q}$-\emph{index} $i_\mathbb{Q}(\folF)$ is the smallest positive integer $m$ such that $mK_\folF$ is Cartier at the $\mathbb{Q}$-Gorenstein points. 
\\

In order to study the Hilbert function $P(m)=\chi(X,mK_\folF)$ of a foliated surface $(X,\folF)$ by using Theorem \ref{thm:rr}, it is necessary to compute the terms $a(x,K_\folF)$ for all singular points $x$ of $(X,\folF)$. For canonical singularities, this was fully done in \cite{langer21}, so we recall the relevant results in the following proposition. The computations rely on a formal description of terminal and canonical singularities, for which we refer to \cite[Corollary I.2.2 and Fact I.2.4]{mcquillan08}.

\begin{proposition}[{\cite[Section 2]{langer21}}]
\label{prop:lansings}
Let $x$ be a terminal or canonical foliation singularity of a foliated surface $(X,\folF)$. Then:
\begin{itemize} 
\item{if $x$ is a terminal singularity, then $a(x, K_\folF) = -\frac{n-1}{2n}$, where $n$ is the index of $x$;}
\item{if $x$ is a canonical non-terminal $\mathbb{Q}$-Gorenstein singularity, then either $a(x,mK_\folF)=0$ for all $m$, or $a(x,mK_\folF) = -\frac{1}{2}$ for odd $m$, and $0$ otherwise;}
\item{If $x$ is a canonical singularity such that $\folF$ is non-$\mathbb{Q}$-Gorenstein at $x$, then $a(x,mK_\folF) = -1$ for $m>0$ and $0$ for $m=0$.}
\end{itemize}
\end{proposition}

\subsection{Canonical models of foliated surfaces}

\begin{definition}[\cite{mcquillan08}, Definition III.3.1]
\label{def:canmod}
A foliated surface $(X,\folF)$ is called a \emph{canonical model} if $X$ is normal, $\folF$ only has canonical singularities, $K_\folF$ is nef, and for all irreducible curves $C$, $K_\folF \cdot C  = 0$ implies $C^2 \geq 0$. 
\end{definition}

Note that in the definition, intersections and nefness are defined as in Section \ref{subsection:inters} because $K_\folF$ might not be $\mathbb{Q}$-Cartier (\cite{mcquillan08}, Theorem IV.2.2). For the same reason, when $(X,\folF)$ is a canonical model of general type, $K_\folF$ might not be ample even though the next lemma shows that it numerically behaves like an ample divisor.

\begin{lemma}
\label{lemma:canample}
    If $(X,\folF)$ is a canonical model of general type, then $K_\folF^2 > 0$ and $K_\folF \cdot C >0$ for every irreducible curve $C$ on $X$. 
\end{lemma}
\begin{proof}
    
    We prove the statement by contradiction. Suppose there exist a curve $C$ such that $K_\folF \cdot C =0$, then by the Hodge index theorem (Lemma \ref{lemma:hodge}),
    $$K_\folF^2 C^2 \leq (K_\folF \cdot C)^2 = 0.$$
    Since $K_\folF$ is big and nef, then $K_\folF^2 > 0$ and $C^2 \leq 0$, which means that $C^2=0$ by the definition of canonical model. Again by the Hodge index theorem, the class of $C$ must be proportional to the class of $K_\folF$, so the only possibility is that $C$ is numerically trivial. 
    
    Let $f \colon (X_m,\folF_m) \rightarrow (X,\folF)$ be the minimal resolution of the non-$\mathbb{Q}$-Gorenstein singularities of $(X,\folF)$. By \cite[Theorem IV.2.1.]{mcquillan08} write $f^*K_\folF=K_{\folF_m}=H+Z$, where $H$ is an ample $\mathbb{Q}$-divisor and Z is an effective $\mathbb{Q}$-divisor supported on the exceptional locus of $f$. Then,
    $$K_\folF \cdot C = K_{\folF_m} \cdot f^{-1}_* C = 0 $$ by the projection formula, but $$K_{\folF_m} \cdot f^{-1}_* C = (H+Z) \cdot f^{-1}_* C > 0,$$ because $H$ is ample and $Z$ is effective. This gives the desired contradiction.
\end{proof}

\subsection{Minimal partial du Val resolutions}
Since a canonical model $(X,\folF)$ is not necessarily $\mathbb{Q}$-Gorenstein, it is possible to pass to the projective setting by considering partial resolutions such that the pullback foliation is $\mathbb{Q}$-Gorenstein. This is one of the reasons to consider the \emph{minimal partial du Val resolution} of a canonical model \cite[Definition 3.1]{chen21}.

\begin{definition}
\label{def:mpdvr}
    Let $(X_c,\folF_c)$ be a canonical model of general type, and let $(X^m,\folF^m)$ be the minimal resolution of the canonical non-terminal singularities of $(X_c,\folF_c)$ together with its pullback foliation; let $g \colon (X^m,\folF^m) \rightarrow (X_c,\folF_c)$ be the associated morphism. By running a classical MMP, let $h \colon X^m \rightarrow X$ be the relative canonical model over $X_c$, which is obtained by contracting smooth rational curves $C$ with $C^2 = -2$ in the fibers of $g$ (in particular, $K_X$ is ample over $X_c$), and let $\folF$ be the pushforward foliation on $X$. $(X,\folF)$ is called the \emph{minimal partial du Val resolution} (MPDVR) of $(X_c,\folF_c)$.
\end{definition}

\section{Finiteness of Hilbert functions of canonical models}
\begin{theorem}
\label{theorem:main}
Let $k_1,k_2$ be rational numbers, $s$ a positive integer. Let $\mathcal{H}_{k_1,k_2,s}$ be the set of Hilbert functions $P(m) = \chi(X, mK_\folF)$ of canonical models $(X,\folF)$ of general type such that $K_\folF^2=k_1$, $K_\folF \cdot K_X=k_2$ and $i_\mathbb{Q}(\folF)=s$. Then $\mathcal{H}_{k_1,k_2,s}$ is finite.
\end{theorem}

\begin{proof}
Let $(X_c, \folF_c) \in \mathcal{H}_{k_1,k_2,s}$ be a foliated canonical model, and let $f \colon (X,\folF) \rightarrow (X_c,\folF_c)$ be its minimal partial du Val resolution. By \cite[Theorem 5]{langer21}, we know that $R^1f_*\bigO_X(mK_\folF)=0$, hence $H^i(mK_\folF)=H^i(mK_{\folF_C})$ for all $m \geq 0$, and in particular $\chi(mK_\folF)=\chi(mK_{\folF_c})$. As a consequence, by Theorem \ref{thm:rr} we also have $K_\folF^2=K_{\folF_c}^2$ and $K_\folF \cdot K_X = K_{\folF_c} \cdot K_{X_c}$; furthermore, $i(\folF)=i_\mathbb{Q}(\folF)=i_\mathbb{Q}(\folF_c)$ by the construction of  the minimal partial du Val resolution. Therefore, in order to prove the statement it is enough to show that it holds for minimal partial du Val resolutions instead.

Let $E= \sum E_i$ be the exceptional divisor of $f$, and let $D = 4i(\folF)K_\folF+K_X$. We note that $D$ is ample: to see this, by the Nakai-Moishezon criterion it is enough to check that $D$ is big, and the intersection of $D$ with every irreducible curve is strictly positive. Since $K_\folF$ is big and $D=K_\folF + (3i(\folF)K_\folF+K_X)$, $D$ is big as well if $3i(\folF)K_\folF+K_X$ is nef.

Let $C$ an irreducible curve on $X$:
\begin{itemize}
\item{$C=E_i$: in this case, 

$$D \cdot C = (4i(\folF)K_\folF+K_X) \cdot C = K_X \cdot C > 0,$$
as by construction $K_X$ is ample over $X_c$.}

\item{$K_X \cdot C \geq 0$: then,

$$D \cdot C \geq 4i(\folF)K_\folF \cdot C = 4i(\folF)K_{\folF_c} \cdot f_* C >0,$$
because $K_\folF = f^* K_{\folF_c}$ and $K_{\folF_c}$ is numerically ample.}

\item{$K_X \cdot C < 0$: by \cite[Theorem 3.8]{fuj12}, every $K_X$-negative extremal ray is spanned by a rational curve with $-3 \leq K_X \cdot C < 0$, so 
$$D \cdot C \geq 4i(\folF)K_\folF \cdot C -3 > 0.$$}

\end{itemize}

Thus, $D \cdot C >0$ for any irreducible curve $C$; furthermore, the same argument shows that $3i(\folF)K_\folF + K_X$ is nef, so $D$ is ample. Since $i(\folF)=s$ is fixed and $i(K_X)\mid i(\folF)$ (because by construction, foliated non-terminal singularities of $(X,\folF)$ are du Val singularities), $sD$ is an ample Cartier divisor and for $m \gg 0$, $\chi(X,msD) = h^0(X, msD)$. So we can apply Theorem \ref{theorem:km} to say that for $m \gg 0$,

$$|P(msD)-\frac{m^2s^2 D^2}{2}| \leq Q(m),$$
where $Q(m)$ is a linear polynomial only depending on $s^2D^2$ and $s(D \cdot K_X)$. To be more precise, the previous inequality holds for any $m>0$: in fact, by construction $sD$ is big and nef, hence the Kawamata-Viehweg theorem gives $H^i(X,msD)=0$ for all $i>0$, that is $h^0(X,msD)=P(ms)$.

Next, we show that $D^2$ and $D \cdot K_X$ can only assume a finite number of values. In particular, since $K_\folF^2$ and $K_\folF \cdot K_X$ are fixed, we need to prove that $K_X^2$ has only a finite number of possible values. 

From Lemma \ref{lemma:hodge},
$$K_\folF ^2 K_X^2 \leq (K_\folF \cdot K_X)^2,$$
and since $K_\folF^2$ and $K_\folF \cdot K_X$ are fixed, $K_X^2$ is bounded from above. On the other hand, since $D$ is ample, $D^2>0$. But
$$D^2 = 16i(\folF)K_\folF^2 + 8i(\folF)K_\folF \cdot K_X + K_X^2>0,$$
which implies that $K_X^2$ is bounded from below; since $i(K_X)$ is bounded, this shows that there are only finitely many values for $K_X^2$. Then, we can suppose $K_X^2$ is fixed, so that both $D^2$ and $D \cdot K_X$ are fixed. Since $Q(m)$ only depends on $D^2$ and $D \cdot K_X$, $P(ms)-\frac{m^2s^2K_\folF^2}{2}$ is bounded by the same polynomial $Q(m)$ for infinite values of $m$ and for all partial resolutions $(X,\folF)$, which means that the set of possible polynomials equal to $P(ms)$ is finite. These only differ for the constant term $\chi(\bigO_X)$, hence the set of values of $\chi(\bigO_X)$ is finite as well. From Theorem \ref{thm:rr}, if we fix $\chi(\mathcal{O}_X)$ too, $P(m)$ is determined up to the term $\sum a(x,mK_\folF)$; each $a(x,mK_\folF)$ can only assume a finite number of values by Proposition \ref{prop:lansings} as $i(\folF)$ is bounded, so we only need to show that the number of singularities is bounded. 

Since $\chi(X,msK_\folF)$ is a polynomial in $m$ belonging to a finite set, \cite[Theorem 2.1.2]{kollar85}, implies that the family of polarised surfaces $(X, sK_\folF)$ is bounded; in particular, the surfaces $X$ belong to a bounded family $f \colon \mathcal{X} \rightarrow \mathcal{T}$. Since normality is an open condition, we can restrict the family and suppose that every fiber of $f$ is normal. By generic smoothness, $f$ is smooth outside a closed set $K \subset \mathcal{X}$, where $K=\bigcup K_i$ and each $K_i$ is irreducible. Consider the restriction 
$\restr{f}{Ki} \colon K_i \rightarrow \mathcal{T}$; since the fibers of $f$ are normal, every fiber of $\restr{f}{K_i}$ is a finite set and $\restr{f}{K_i}$ is quasi-finite. Furthermore, since $f$ is proper $\restr{f}{K_i}$ is proper as well; then, $f$ is finite and the cardinality of each fiber is bounded by the degree of $\restr{f}{K_i}$. This implies that the number of singularities on the fibers is bounded from above by $\sum \deg (\restr{f}{K_i})$.
\end{proof}

Thanks to Theorem \ref{theorem:main}, all the boundedness results of \cite{langer21} and \cite{chen21} that require the Hilbert function to be fixed still hold under the weaker hypotheses of Theorem \ref{theorem:main}. In particular, from Theorems \ref{thm:langer} and \ref{thm:chen} we deduce the following.

\begin{corollary}
\label{cor:hl}
Fix rational numbers $k_1,k_2$ and a positive integer $s$. There exists a constant $N_1$, only depending on $k_1,k_2,s$, such that for any canonical model $(X,\folF)$ with $K_\folF^2 = k_1$, $K_\folF \cdot K_X = k_2$ and $i_\mathbb{Q}(\folF)$, and for any $m\geq N_1$, $|mK_\folF|$ defines a birational map.
\end{corollary}

\begin{corollary}
\label{cor:chen}
    Fix rational numbers $k_1,k_2$ and a positive integer $s$. The set $\mathcal{S}_{k_1,k_2,s}$ of minimal partial du Val resolutions of canonical models of general type $(X,\folF)$  with fixed $K_\folF^2=k_1, K_\folF \cdot K_X = k_2, i_\mathbb{Q}(\folF)=s$ is bounded.
\end{corollary}

\section{Examples and particular cases}
In this section we collect further results and examples with two goals in mind: to further understand the properties of bounded families of canonical models, and to present particular cases in which fixing $i_\mathbb{Q}(\folF)$ is redundant. From here on, we will only consider canonical models which do not have cusp singularities (hence $K_\folF$ is $\mathbb{Q}$-Cartier).

\subsection{An unbounded family with fixed $K_\folF^2$}
We begin with an example of an unbounded family of canonical models with fixed volume $K_{\mathcal{F}}^2$, but unbounded $K_\folF \cdot K_X$ \cite[Example 2]{xiao87}.

\begin{example}
\label{ex:kf2}
Let $C$ be a smooth curve, $k$ an even integer, $g\geq 2$ an integer. Let $D$ a divisor on $C$ of degree $k$ such that $|(2g+1)D|$ is basepoint free, consider the ruled surface $P = \mathbb{P}(\bigO_C \oplus \bigO_C(D))$ over $C$, and let $\pi$ be the projection on $C$. Let $C_0,C_1$ be disjoint global sections of $\pi$ such that $C_0^2 = -k$, $C_1^2 = k$, $C_0 \cdot C_1 = 0$. 

Let $\Lambda = |(2g+1)C_1|$: the system is basepoint free, hence by Bertini's theorem $\Lambda$ contains a divisor $B$ which is irreducible, reduced and also smooth. As a consequence, since $B$ and $C_0$ are disjoint, the divisor $R = B + C_0$ is smooth and reduced, and in $\mathrm{Pic}(P)$ $R = (2g+1)kF + (2g+2)C_0$ is divisible by 2, where $F$ is a general fiber of $P$; we can then consider the double cover $\sigma: S \rightarrow P$, ramified along $R$. By composition, we get a new fibration $f: S \rightarrow C$ which induces a foliation $\folF$ on $S$ such that $K_\folF = K_{S/C}$. A computation shows that the genus of the fibers is $g$, and $K_\mathcal{F}^2$ is given by

\begin{equation*}
\begin{split}
K_\mathcal{F}^2 = 2kg(g-1),
\end{split}
\end{equation*}

which is independent of the genus of the base curve $C$; furthermore, $(S, K_\folF)$ is a canonical model. Then, if we take $C$ to be a curve of arbitrarily large genus admitting a divisor $D$ as in the construction (for example, any hyperelliptic curve), since the curves $C$ are not bounded the surfaces $S$ are unbounded as well.

\end{example}

Note that in this example, by \cite[Lemma 3.5]{langer21}, $m_C K_\folF - K_X$ is pseudoeffective for an integer $m_C>0$ that increases with the genus of $C$. This can be seen as the reason why boundedness fails, as shown by the following result.

\begin{proposition}
    Let $\mathcal{M}_{k,s}$ the collection of $\mathbb{Q}$-Gorenstein foliated canonical models $(X,\folF)$ such that $K_\folF^2 = k$, $i(\folF)=s$. Then, the following are equivalent.
    \begin{itemize}
        \item{The set \{$K_\folF \cdot K_X \mid (X,\folF)\in \mathcal{M}_{k,s}\}$ is finite.}
        \item{There exists an integer $m>0$ such that $mK_\folF-K_X$ is pseudoeffective for any $(X,\folF)\in \mathcal{M}_{k,s}.$}
    \end{itemize}
\end{proposition}

\begin{proof}
    If $K_\folF \cdot K_X$ only admits a finite number of values, we can suppose that it is fixed. By the proof of \cite[Theorem 3.6]{langer21}, $3i(\folF)K_\folF + K_X$ is nef, and by \cite[Lemma 3.5]{langer21} there is a constant $m$, only depending on $K_\folF^2,$ $ K_\folF \cdot K_X$ and $i(\folF)$, such that $mK_\folF - K_X$ is pseudoeffective.

    Conversely, suppose that there exists an integer $m>0$ such that $mK_\folF - K_X$ is pseudoeffective for all $(X,\folF) \in \mathcal{M}_{k,s}$. Since $K_\folF$ is ample, $(mK_\folF - K_X) \cdot K_\folF \geq 0$, which implies that $K_\folF \cdot K_X \leq mK_\folF^2$; on the other hand, as $3i(\folF)K_\folF + K_X$ is nef, $(3i(\folF)K_\folF + K_X)\cdot K_\folF \geq 0$ which implies that $K_\folF \cdot K_X \geq -3i(\folF) K_\folF$. Since $i(\folF)$ is fixed, $K_\folF \cdot K_X$ can assume only a finite number of values. 

    \end{proof}

\subsection{A bounded family with unbounded genus}

Next, we present an example of a collection of algebraically integrable foliations that belong to a bounded family even though the leaves have unbounded genus. 

\begin{example}
Let $E$ be an elliptic curve, $P \coloneqq E \times E$. We can consider two different morphisms of $P$ onto $E$: besides the coordinate projections (we call $\pi_x$ and $\pi_y$ the projections onto the first and second coordinate, respectively), the $n$-multiplication map on $E$, $[n] \colon x \mapsto n \cdot x$, allows us to view its graph $\Gamma_n = \{(x,n \cdot x) \mid x \in E \}$ as a subvariety of $P$, isomorphic to $E$. Then, all the translations $(0, y_0) + \Gamma_n$ of $\Gamma_n$ form a family of disjoint curves, isomorphic to $E$ and covering $P$. Thus, we get another projection $P \rightarrow E$, defined by sending a point $(x,y)$ to $y-n \cdot x$; we call this projection $\pi_n$. 

For a suitable divisor D of degree $d>1$ on $E$, let $A\sim \pi_x^*(2D)+\pi_y^*(2D)$ be a very ample divisor on $P$, which we can choose smooth, reduced and irreducible by Bertini's theorem. Let $S$ be the double cover of $P$ ramified along $A$, $\sigma \colon S \rightarrow P$ the covering map, and $f_n$ the composition $\pi_n \circ \sigma$. We have that $K_S\sim_{\mathbb{Q}}\sigma_n^*(K_P+\frac{1}{2}A)$; in this case, since both $K_E$ and $K_P$ are trivial, we get that $K_{S/E} = K_{S} = \sigma^*(\frac{1}{2}A)$. We consider the foliation $\mathcal{F}_n$ induced by $f_n$: as both $P$ and $A$ are smooth, both $S$ and the fibers of $f_n$ are smooth as well, hence $\folF_n$ is regular and has canonical singularities. The fibers are reduced, so that $K_{\mathcal{F}_n} = K_{S/E}$; then, $K_{\mathcal{F}_n}^2 =  K_{\mathcal{F}_n} \cdot K_S = (\sigma^*(\frac{1}{2}A))^2= \frac{1}{2}A^2$. Note that the genus of the fibers of $f_n$ depends on $n$: in fact, if $F$ is a fiber of $\pi_n$, we have that $A \cdot F = 2d(n^2+1)$. Then, the Riemann-Hurwitz formula implies that the genus of a fiber $F_n$ of $f_n$ is equal to $d(n^2+1)+1$. Still, the family of foliated surfaces is bounded because $K_{\folF_n}^2=K_{\folF_n}\cdot K_S$ is independent of $n$, and $S$ is smooth.

\end{example}

\subsection{Special types of surfaces}
Although it is reasonable to expect that Theorem \ref{theorem:main} does not hold without the condition on $i(\folF)$, it is automatically satisfied under some assumptions on the underlying surfaces. A first example of this is given by surfaces with big and nef anticanonical divisor: while a priori they are unbounded even if their volume $K_X^2$ is fixed, their index is bounded when $K_\folF^2$ and $K_\folF \cdot K_X$ are fixed.
    \begin{theorem}
    \label{thm:i1}
        Let $\mathcal{M}^{N}_{k_1,k_2}$  be the collection of $\mathbb{Q}$-Gorenstein canonical models of general type $(X,\folF)$ such that $-K_X$ is big and nef, $K_\folF^2=k_1$, $K_\folF \cdot K_X = k_2$. Then, there exists an integer $N_F>0$ such that $i(\folF) \leq N_F$ for all $(X,\folF) \in \mathcal{M}^{N}_{k_1,k_2}$.
    \end{theorem}
    
    \begin{proof} Since $(X,\folF)$ is a $\mathbb{Q}$-Gorenstein canonical model, $X$ is klt. Then, as $-K_X$ is big and nef, the Kawamata-Viehweg theorem implies that  $$\chi(K_\folF)=\chi(K_X+(K_\folF-K_X))=h^0(K_\folF),$$ and $$\chi(\mathcal{O}_X)=\chi(K_X+(-K_X))=1.$$ Hence, from Theorem \ref{thm:rr},
    $$-\sum_{x \in \text{Sing}X} a(x,K_\folF) = \frac{1}{2}K_\folF \cdot(K_\folF - K_X) - h^0(K_\folF)+1.$$
    Since $h^0(K_\folF)$ is a nonnegative integer, $-\sum a(x,K_\folF)$ is bounded from above and it can assume only a finite number of values, as it is nonnegative by Proposition \ref{prop:lansings}. Then, we can suppose that $-\sum a(x,K_\folF)$ is fixed, so that we can argue as in \cite[Proposition 4.1]{langer21} to show that the number of non-Gorenstein singularities of $\folF$ is bounded: let $\Sigma = \Sigma_1 \cup \Sigma_2$ be the set of singular points of $\folF$, where $\Sigma_1$ are terminal singularities and $\Sigma_2$ are the dihedral quotient singularities. Then,
    $$-\sum_{x \in \Sigma} a(x,K_\folF) = \sum_{x\in \Sigma_1} \frac{n_x-1}{2n_x}+\sum_{x \in \Sigma_2} \frac{1}{2} \geq \frac{1}{2}|\Sigma|,$$
    where $n_x$ is the index of the cyclic quotient singularity at $x$. This shows that $|\Sigma|$ is bounded, which implies that
    $$\sum_{x \in \Sigma_1} \frac{1}{n_x} = |\Sigma| + 2\sum_{x \in \Sigma}a(x,K_\folF)$$
    can assume only finitely many values. By \cite[Lemma 3.4]{langer21}, then $n_x$ must be bounded.

    \end{proof}

The second class of cases is given by surfaces which can be shown to belong to a bounded family under the numerical assumptions on $K_\folF$.

\begin{theorem}
\label{thm:i2}
Fix $k_1,k_2$. Then, the $\mathbb{Q}$-Gorenstein canonical models $(X,\folF)$ of general type with $K_\folF^2 = k_1, K_\folF \cdot K_X= k_2$, whose underlying surface belongs to one of the following, are bounded.
\begin{enumerate}[label=\textnormal{(\roman*)}]
\item{Surfaces with $K_X \equiv 0$.}
\item{Algebraically integrable foliations induced by fibrations with reduced fibers.}
\end{enumerate}
\end{theorem}
\hfill
\begin{proof}
Note that in the following, $X$ is again klt because $(X,\folF)$ is $\mathbb{Q}$-Gorenstein.

\begin{enumerate}[label=\textnormal{(\roman*)}]

    \item{Since $K_\folF$ is big and nef and $K_\folF^2$ is fixed, the result follows from \cite[Corollary 1.6]{birkar23}. Equivalently, it can be seen that in this case $(X,\folF)$ is an $\epsilon$-adjoint canonical model (\cite[Definition 3.6]{spicer22}) with $\delta=1$, then the result follows from \cite[Theorem 6.1]{spicer22}.}

    \item{ Let $f \colon X \rightarrow C$ be a fibration with reduced fibers, $\folF$ the induced foliation, hence $K_\folF=K_{X/C}=K_X-f^*K_C$. Let $F$ be a general fiber of $f$, then
$$K_\folF^2=K_{X/C}^2=K_X^2-8(g(F)-1)(g(C)-1),$$
and
$$K_\folF \cdot K_X=K_X^2-4(g(F)-1)(g(C)-1).$$
Since $K_\folF^2$ and $K_\folF \cdot K_X$ are constant, $K_X^2$ is constant as well, because $$K_X^2= 2(K_\folF \cdot K_X) - K_\folF^2.$$ 
If $g(C)\geq 1$, $K_X$ is ample as well, hence $X$ is a canonical model with fixed volume $K_X^2$. When $g(C)=0$, consider the linear system $|-f^*K_C|$, which is basepoint free. Let $D_1,D_2 \in |-f^*K_C|$, $D_1=F_1+F_2$, $D_2= F_3+F_4$ be two general members with $F_1,\ldots,F_4$ distinct fibers of $f$, so that $-f^*K_C \sim_\mathbb{Q} \frac{1}{2}(D_1+D_2)$. Let $D = \frac{1}{2}(D_1+D_2)$, then by \cite[Lemma 5.17]{km98}, $(X,D)$ is klt with $(K_X+D)^2=K_\folF^2$ fixed. In both cases, the underlying surfaces are bounded by \cite[Theorem 9.2]{ale94}.
}
\end{enumerate}

\end{proof}

\end{document}